\documentclass[12pt]{article}%
\usepackage{ytableau}
\usepackage{amsmath}
\usepackage{amsfonts}
\usepackage{amssymb}

\newtheorem{theorem}{Theorem}

\newtheorem{lemma}[theorem]{Lemma}

\newenvironment{proof}[1][Proof]{\noindent\textbf{#1.} }{{\hfill $\Box$ \\}}
\begin{document}

\title{On Zel'manov's global nilpotence theorem for Engel Lie algebras}
\author{Michael Vaughan-Lee}
\date{July 2025}
\maketitle

\begin{abstract}
I give a proof of Zel'manov's theorem that if $L$ is an $n$-Engel Lie algebra
over a field $F$ of characteristic zero, then $L$ is (globally) nilpotent.
This is a very important result which extends Kostrikin's theorem that $L$ is
locally nilpotent if the characteristic of $F$ is zero or some prime $p>n$.
Zel'manov's proof contains some striking original ideas, and I wrote this
note in an effort to understand his arguments. I hope that my efforts will
be of use to other mathematicians in understanding this remarkable theorem.
I am grateful to Christian d'Elb\'{e}e for a number of helpful comments on
earlier drafts of this note.

\end{abstract}

\section{Introduction}

Efim Zel'manov \cite{zelmanov87} proved that if $L$ is an $n$-Engel Lie
algebra over a field $F$ of characteristic zero, then $L$ is (globally)
nilpotent. This result extends Kostrikin's theorem that if $F$ has
characteristic zero or prime characteristic $p>n$, then $L$ is locally
nilpotent (see \cite{kostrikin59}, \cite{kostrikin86}). Kostrikin gives an
expanded version of Zel'manov's proof in his book \emph{Around Burnside
}\cite{kostrikin86}, but I feel there is room for another version of the
proof of this remarkable result. Accordingly I present here a proof of
Zel'manov's theorem which actually closely follows his original proof, though
it also draws on Kostrikin's presentation of the proof in \cite{kostrikin86}.

However my proof does differ from the proofs in \cite{zelmanov87} and
\cite{kostrikin86} in one key point. The starting point for Zel'manov's proof
is Kostrikin's theorem (\cite{kostrikin59}, \cite{kostrikin86}) that every
(non-zero) $n$-Engel Lie algebra over a field $F$ of characteristic zero or
prime characteristic $p>n$ contains a non-zero abelian ideal. Zel'manov uses
this to show that every $n$-Engel Lie algebra of characteristic zero is the
union of an ascending chain of ideals defined as follows. He sets
$I_{0}=\{0\}$, and proceeds by transfinite induction. If $\alpha$ is a limit
ordinal he sets $I_{\alpha}=\cup_{\nu<\alpha}I_{\nu}$, and if $\alpha$ is a
successor ordinal and $I_{\alpha-1}$ is defined he sets $I_{\alpha}$ to be the
inverse image in $L$ of the sum of all abelian ideals in $L/I_{\alpha-1}$.
However Kostrikin's proof that every non-zero $n$-Engel Lie algebra over a
field $F$ of characteristic zero or prime characteristic $p>n$ contains a
non-zero abelian ideal actually shows that $L$ has a finite chain of ideals%
\[
L=I_{0}>I_{1}>\ldots>I_{k-1}>I_{k}=\{0\}
\]
with the property that $I_{j}/I_{j+1}$ is the sum of all abelian ideals of
$L/I_{j+1}$ for $j=0,1,\ldots,k-1$. Adjan and Razborov \cite{adjanraz87} show
that the length $k$ of this chain can be bounded by $N(n,10).6^{n+12}$ where
$N(n,r)$ is defined by%
\[
N(n,4)=6,\;N(n,r+1)=F(n,r+1,N^{2}(n,r).3^{(n+6)/2}),
\]
and where $F$ is defined by
\[
F(n,r,0)=1,\;F(n,r,i+1)=n.r^{3F(n,r,i)}.
\]

So we let $L$ be the free $n$-Engel Lie algebra of countably infinite rank
over a field $F$ of characteristic zero, and we let%
\[
L=I_{0}>I_{1}>\ldots>I_{k-1}>I_{k}=\{0\}
\]
be a finite chain of ideals as described above. (The bound on $k$ does not
concern us.)

We prove by reverse induction on $j$ that the ideal $I_{j}$ is fully
invariant. Clearly the ideal $I_{k}$ is fully invariant. So suppose that
$I_{j+1}$ is fully invariant, and consider the ideal $I_{j}/I_{j+1}$ in
$L/I_{j+1}$. Since $I_{j+1}$ is fully invariant, $L/I_{j+1}$ is relatively
free, and any relation which holds in $L/I_{j+1}$ is actually an identical
relation. Let $M=L/I_{j+1}$ and let $I=I_{j}/I_{j+1}$. So $I$ is the sum of
all abelian ideals of the relatively free Lie algebra $M$. We need to show
that $I$ is a fully invariant ideal of $M$. To this end it is sufficient to
show that if $\theta$ is an endomorphism of $M$ and if $a$ lies in an abelian
ideal of $M$ then the ideal of $M$ generated by $a\theta$ is abelian. We let
the free generators of $M$ be $x_{1},x_{2},\ldots$ and we let $a$ lie in the
subring of $M$ generated by $x_{1},x_{2},\ldots,x_{r}$. Since the ideal
generated by $a$ is abelian%
\[
\lbrack a,x_{r+1},x_{r+2},\ldots,x_{r+m},a]=0
\]
is an identical relation in $M$ for all $m>0$. Now let $a_{1},a_{2}%
,\ldots,a_{m}$ be arbitrary elements of $M$ and let $\varphi$ be an
endomorphism of $M$ such that $x_{i}\varphi=x_{i}\theta$ for $i=1,2,\ldots,r$
and such that $x_{r+i}\varphi=a_{i}$ for $i=1,2,\ldots,m$. Then%
\[
\lbrack a\theta,a_{1},a_{2},\ldots,a_{m},a\theta]=[a,x_{r+1},x_{r+2}%
,\ldots,x_{r+m},a]\varphi=0,
\]
which shows that the ideal of $M$ generated by $a\theta$ is abelian, as claimed.

We show by induction on the length of this series that $L$ is nilpotent.
Our base inductive step is to show that $L/I_1$ is nilpotent using the fact
that $L/I_1$ is a relatively free $n$-Engel Lie algebra which is a sum of
abelian ideals. For the general inductive step we suppose that $L/I_j$ is nilpotent
for some $j$ with $1\leq j<k$, and we prove that this implies that $L/I_{j+1}$ is
nilpotent.

Our base inductive step is very easy. Let $M=L/I_1$. Then $M$ is a relatively free 
$n$-Engel Lie algebra over $F$, and $M$ is sum of abelian ideals.
If $x$ is a free generator of $M$ then we can write 
\[
x=a_1+a_2+\ldots +a_{k}
\]
for some $k$, where the elements $a_1,a_2,\ldots ,a_{k}$ all lie in abelian ideals of $M$.
Since $M$ is relatively free, this relation is an identical relation. Let $\theta $ be the
endomorphism of $M$ which maps $x$ to $x$, and maps all other free generators of $M$ to zero.
Then
\[
x=x\theta =a_1\theta +a_2\theta +\ldots +a_{k}\theta.
\]
The elements $a_i\theta $ must all be scalar multiples of $x$, and as we showed above they
must all lie in abelian ideals of $M$. So $x$ lies in an abelian ideal of $M$, which
implies that $M$ satisfies the 2-Engel identity $[y,x,x]=0$. It is well known that the
2-Engel identity in characteristic zero implies the identity $[x,y,z]=0$. (See Example 2.7 from
the introduction to \cite{kostrikin86}, or Theorem 3.1.1 of \cite{vl1993}.) 
So $M$ is nilpotent of class 2.

For the general inductive step we need some of the representation theory of the symmetric group.

\section{Representation theory of the symmetric group}

We let $N$ be a positive integer, and we consider the group ring $\mathbb{Q}%
$Sym$(N)$ of the symmetric group on $N$ letters, where $\mathbb{Q}$ is the
rational field. The identity element in $\mathbb{Q}$Sym$(N)$ is a sum of
primitive idempotents, and these are described in James and Kerber
\cite{jamesker81}: they correspond to \emph{Young tableaux}. For each
partition $(m_{1},m_{2},\ldots,m_{s})$ of $N$ with $m_{1}\geq m_{2}\geq
\ldots\geq m_{s}$ we associate a \emph{Young diagram}, which is an array of
$N$ boxes arranged in $s$ rows, with $m_{i}$ boxes in the $i$-th row. The
boxes are arranged so that the $j$-th column of the array consists of the
$j$-th boxes out of the rows which have length $j$ or more. For example, if
$N=5$ there are seven possible Young diagrams.

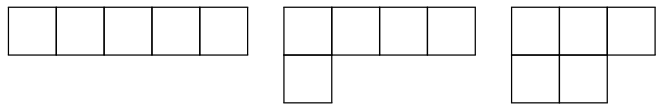
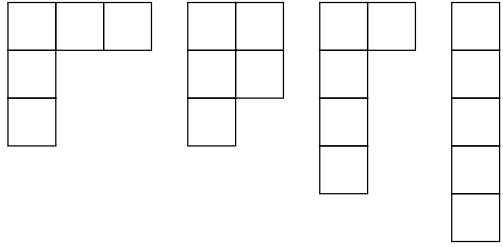
\[
\ydiagram{5}\;\;\;\;
\ydiagram{4,1}\;\;\;\;
\ydiagram{3,2}
\]
\bigskip

\[
\ydiagram{3,1,1}\;\;\;\;
\ydiagram{2,2,1}\;\;\;\;
\ydiagram{2,1,1,1}\;\;\;\;
\ydiagram{1,1,1,1,1}
\]
\bigskip

We obtain a Young tableau from a Young diagram by filling in the $N$ boxes
with $1,2,\ldots,N$ in some order. We then let $H$ be the subgroup of Sym$(N)$
which permutes the entries within each row of the tableau, and we let $V$ be
the subgroup of Sym$(N)$ which permutes the entries within each column of the
tableau. We set%
\[
e=\sum_{\pi\in V,\,\rho\in H}\text{sign}(\pi)\pi\rho.
\]
Then $\frac{1}{k}e$ is a primitive idempotent of $\mathbb{Q}$Sym$(N)$ for some
$k$ dividing $N!$. As mentioned above the identity element in $\mathbb{Q}%
$Sym$(N)$ can be written as a sum of primitive idempotents of this form, and
if the field $F$ has characteristic zero then so can the identity element in
$F$Sym$(N)$.

One key property of these Young tableaux is that if we have a Young tableau on
$N$ letters, then either the first row or the first column of the tableau must
have length at least $N^{\frac{1}{2}}$.

\section{A key lemma}

\begin{lemma}
Let $L=L_{0}\oplus L_{1}$ be an $n$-Engel Lie algebra with a $\mathbb{Z}_{2}%
$-grading, and suppose that $L_{0}$ is nilpotent of class at most $m-1$, so
that $[x_{1},x_{2},\ldots,x_{m}]=0$ is an identical relation in $L_{0}$. Then
$L$ is nilpotent of class bounded by $\frac{n^{(n-1)(m-1)+1+m}-1}{n-1}$.
\end{lemma}

\begin{proof}
The $\mathbb{Z}_{2}$-grading on $L$ means that $[L_{0},L_{0}]\leq L_{0}$,
$[L_{1},L_{0}]\leq L_{1}$, $[L_{1},L_{1}]\leq L_{0}$. Note that this grading
does \emph{not} turn $L$ into a Lie superalgebra --- $L$ is an $n$-Engel Lie algebra.

Let%
\[
L=L^{(0)}\geq L^{(1)}\geq L^{(2)}\geq\ldots
\]
be the derived series of $L$, so that $L^{(1)}=[L^{(0)},L^{(0)}]=[L,L]$,
$L^{(2)}=[L^{(1)},L^{(1)}]$, and so on. Then $L^{(1)}\leq L_{0}+[L_{1},L_{0}%
]$, $L^{(2)}\leq L_{0}+[L_{1},L_{0},L_{0}]$, and in general $L^{(k)}\leq
L_{0}+[L_{1},\underset{k}{\underbrace{L_{0},\ldots,L_{0}}}]$. Consider an
element $[b,a_{1},a_{2},\ldots,a_{k}]\in\lbrack L_{1},\underset{k}%
{\underbrace{L_{0},\ldots,L_{0}}}]$, where $b\in L_{1}$ and $a_{1}%
,a_{2},\ldots,a_{k}\in L_{0}$. If $k\geq n$ then Proposition 4.6 of
\emph{Around Burnside }\cite{kostrikin86} implies that $[b,a_{1},a_{2}%
,\ldots,a_{k}]$ is a linear combination of elements $[b,c_{1},c_{2}%
,\ldots,c_{n-1}]$ where $c_{1},c_{2},\ldots,c_{n-1}$ are commutators in
$a_{1},a_{2},\ldots,a_{k}$ whose weights add up to $k$. If $k>(n-1)(m-1)$ then
at least one of the commutators $c_{i}$ must have weight at least $m$, and
must be trivial. So $[L_{1},\underset{k}{\underbrace{L_{0},\ldots,L_{0}}}]=0$
if $k>(n-1)(m-1)$. So $L^{((n-1)(m-1)+1)}\leq L_{0}$, and $L$ is solvable,
with derived length at most $(n-1)(m-1)+1+m$. Higgins's theorem
\cite{higgins54} implies that if $L$ is an $n$-Engel Lie algebra over a field
$F$ of characteristic zero or characteristic $p>n$, and if $L$ is solvable of
derived length $r$ then $L$ is nilpotent with class at most $\frac{n^{r}%
-1}{n-1}$. Let $K$ be the smallest integer which is greater than
$\frac{n^{(n-1)(m-1)+1+m}-1}{n-1}$. Then $L$ is nilpotent of class at most
$K-1$.
\end{proof}

Now let $L$ be a relatively free $n$-Engel Lie algebra with free basis
$x_{1},x_{2},\ldots,x_{K}$ over a field of characteristic zero. 
We turn $L$ into a $\mathbb{Z}_{2}$-graded Lie
algebra $L=L_{0}\oplus L_{1}$ by specifying that some of the free generators
$x_{i}$ are odd and specifying that the remainder are even. We know that $L$ is spanned
by left-normed commutators $c=[x_{i_{1}},x_{i_{2}},\ldots,x_{i_{r}}]$ where 
$r\geq1$ and $i_{1},i_{2},\ldots,i_{r}\in\{1,2,\ldots,K\}$. We can assign a
multiweight $w=(w_1,w_2,\ldots,w_K)$ to $c$ by setting $w_s=|\{i_j:1\le j\le r, \;i_j=s\}|$
for $s=1,2,\ldots ,K$.
In other words, $w_s$ is the degree of $c$ in the free generator $x_s$. For each
possible multiweight $w$ we let $L_w$ be the linear span of all left-normed
commutators with multiweight $w$. Because $L$ is a relatively free Lie algebra
over a field of characteristic zero, $L$ is the direct sum of all these multiweight
components $L_w$. Furthermore, if $u,v$ are two possible multiweights then
$[L_u,L_v] \leq L_{u+v}$, with addition of multiweights defined componentwise. We define $C_0$
to be the set of all left-normed commutators $c$ with multiweight $(w_1,w_2,\ldots ,w_K)$ 
where the sum $\sum_{1\leq s\leq K,\;x_s \,\mathrm{is}\,\mathrm{odd}}w_s$ is even and we
define $C_1$ to be the set of all left-normed commutators $c$ with multiweight 
$(w_1,w_2,\ldots ,w_K)$ where this sum is odd.
If we let $L_0$ be the linear span of $C_0$ and we let $L_1$ be the linear span
of $C_1$ then $L=L_0 \oplus L_1$ is a $\mathbb{Z}_{2}$-graded Lie algebra.

Now let
$I$ be the ideal of $L$ generated by all possible elements $[c_{1}%
,c_{2},\ldots,c_{m}]$ with $c_{i}\in C_{0}$ for $i=1,2,\ldots,m$. Then $L/I$
satisfies the hypothesis of Lemma 1, and so $[x_{1},x_{2},\ldots,x_{K}]\in I$.
For any particular $\mathbb{Z}_{2}$-grading on $L$ this implies that
$[x_{1},x_{2},\ldots,x_{K}]$ is a finite linear combination of terms of the
form $[[c_{1},c_{2},\ldots,c_{m}],a_{1},a_{2},\ldots,a_{t}]$ with $c_{i}\in
C_{0}$ for $i=1,2,\ldots,m$, and with $a_{i} \in \{x_1,x_2,\ldots,x_K\}$ for
$i=1,2,\ldots,t$ ($t \ge 0$). Since $L$ is relatively free we can assume that the 
elements $[[c_{1},c_{2},\ldots,c_{m}],a_{1},a_{2},\ldots,a_{t}]$ all have weight $K$,
and are multilinear in $x_{1},x_{2},\ldots,x_{K}$. We let $T$ be the maximum number 
of elements $[[c_{1},c_{2},\ldots,c_{m}],a_{1},a_{2},\ldots,a_{t}]$ that arise in 
any of these linear combinations as we range over all possible 
$\mathbb{Z}_{2}$-gradings on $L$.

\section{The general inductive step}

As we described in the introduction, to prove Zel'manov's
theorem we need to show that if $L/I_j$ is nilpotent for some $j$ with
$1\leq j<k$ then $L/I_{j+1}$ is nilpotent. So let $M=L/I_{j+1}$ and let $I=I_j/I_{j+1}$,
and assume that $M/I$ is nilpotent. Then $M$ is a relatively free $n$-Engel
Lie algebra over $F$, and $I$ is the sum of all abelian ideals of $M$.
Let the free generators of $M$ be $x_1,x_2,\ldots $. We also let $x_{(i,j)}$
($i,j \ge 1$) denote free generators of $M$. Since $M/I$ is nilpotent,
$[x_1,x_2,\ldots ,x_m] \in I$ for some $m$. So
\[
[x_1,x_2,\ldots ,x_m]=a_1+a_2+\ldots +a_{k-1}
\]
($k>1$) for some elements $a_1,a_2,\ldots ,a_{k-1} \in I$ all of which
lie in abelian ideals of $M$. So $M$ satisfies \emph{all} identical relations
of the form
\[
\lbrack\lbrack x_{1},x_2,\ldots,x_{m}],\ldots ,[x_{1},x_2,\ldots ,x_{m}],\ldots 
,[x_{1},x_2,\ldots,x_{m}]]=0
\]
where there are $k$ occurrences of the commutator $[x_{1},x_{2},\ldots,x_{m}]$
in each of these relations. (To simplify the notation we omit the entries in
the commutator which lie between the entries $[x_1,x_2,\ldots ,x_m]$.)
For each $i=1,2,\ldots,m$ we substitute $\sum_{j=1}%
^{k}x_{(j,i)}$ for $x_{i}$ in these relations. If we expand, and collect
up the terms which are multilinear in $\{x_{(j,i)}\}$ then we see that $M$
satisfies the identical relations%

\begin{equation}
\sum[[x_{(1\sigma_{1},1)},\ldots,x_{(1\sigma_{m},m)}],\ldots,[[x_{(2\sigma
_{1},1)},\ldots,x_{(2\sigma_{m},m)}],\ldots,[[x_{(k\sigma_{1},1)},\ldots
,x_{(k\sigma_{m},m)}]]=0,
\end{equation}
where the sum is taken over all $\sigma_1,\sigma_2,\ldots,\sigma_m \in $ Sym$(k)$.
Let $K$ be the smallest integer which is greater than
$\frac{n^{(n-1)(m-1)+1+m}-1}{n-1}$, as in Section 3, and let $T$ be as defined
in Section 3. Let $N=(Tk)^{2^{K}}$. We show that the identical relation%
\begin{equation}
\lbrack\lbrack x_{(1,1)},x_{(1,2)},\ldots,x_{(1,K)}],[x_{(2,1)},x_{(2,2)}%
,\ldots,x_{(2,K)}],\ldots,[x_{(N,1)},x_{(N,2)}\ldots,x_{(N,K)}]]=0
\end{equation}
is a consequence of the identical relations (1). This implies that $M$ is solvable,
and by Higgins's theorem \cite{higgins54} we can conclude that $M$ is nilpotent.

\bigskip
We let $F$Sym$(N)$ act on $M$, permuting the free generators
$x_{(1,1)},x_{(2,1)},\ldots,x_{(N,1)}$. If $\sigma \in \,$Sym$(N)$ then we let $x_{(i,1)}\sigma = x_{(i\sigma,1)}$ and let $x_{(i,j)}\sigma = x_{(i,j)}$ if $j \ne 1$. To establish equation (2) it is enough to 
show that
\begin{equation}
\lbrack\lbrack x_{(1,1)},x_{(1,2)},\ldots,x_{(1,K)}],[x_{(2,1)},x_{(2,2)}%
,\ldots,x_{(2,K)}],\ldots,[x_{(N,1)},x_{(N,2)}\ldots,x_{(N,K)}]]e=0
\end{equation}
for every primitive idempotent $e$ in $F$Sym$(N)$. A primitive idempotent in
$F$Sym$(N)$ will correspond to a Young tableau with first row of length at
least $N^{\frac{1}{2}}$ or first column of length at least $N^{\frac{1}{2}}$.

Suppose first that $e$ is a primitive idempotent corresponding to a Young
tableau with a row of length at least $N^{\frac{1}{2}}$, and let
\[
e=\frac{1}{m_0}\sum_{\pi\in V,\,\rho\in H}\text{sign}(\pi)\pi\rho
\]
where $V$ is the subgroup of Sym$(N)$ which permutes the entries within
each column of the tableau, and $H$ is the subgroup of Sym$(N)$ which
permutes the entries within each row of the tableau. Pick out the first $N^{\frac{1}{2}}$ 
entries in the first row of the tableau, and arrange them in ascending order 
$i_{1}<i_{2}< \ldots <i_{N^{\frac{1}{2}}}$.
Let $G$ be the subgroup of $H$ which fixes
$\{1,2,\ldots,N\}\backslash\{i_{1},i_{2},\ldots,i_{N^{\frac{1}{2}}}\}$ and $C$ be a left
transversal for $G$ in $H$, so that $H=\cup_{c\in C}cG$. Let $f=$
$\sum_{\sigma\in G}\sigma$. Also let
\[
t_i=[x_{(i,1)},x_{(i,2)}\ldots,x_{(i,K)}]
\]
for $i=1,2,\ldots,N$. Then
\[
\lbrack\lbrack x_{(1,1)},x_{(1,2)},\ldots,x_{(1,K)}],\ldots ,[x_{(N,1)},x_{(N,2)}\ldots,x_{(N,K)}]]e
\]
is a linear combination elements of the form $[t_1,t_2,\ldots,t_N]\pi cf$
with $\pi\in V$ and $c\in C$. For fixed $\pi\in V$ and $c\in C$ let
\[
\{j_1\pi c,j_2\pi c,\ldots ,j_{N^{\frac{1}{2}}}\pi c\}=\{i_{1},i_{2},\ldots,i_{N^{\frac{1}{2}}}\}
\]
with $j_1<j_2<\ldots <j_{N^{\frac{1}{2}}}$. Then $[t_1,t_2,\ldots,t_N]\pi cf$ equals
\[
\sum_{\sigma\in G}[t_1\pi c\sigma,\ldots,t_{j_1}\pi c\sigma,\ldots,t_{j_{N^{\frac{1}{2}}}}\pi c\sigma,
\ldots,t_N\pi c\sigma].
\]
Now $t_i\pi c\sigma = t_i\pi c$ if $i \notin \{j_1,j_2,\ldots,j_{N^{\frac{1}{2}}}\}$, and if
$i \in \{j_1,j_2,\ldots,j_{N^{\frac{1}{2}}}\}$ then
\[
t_i\pi c\sigma=[x_{(i\pi c\sigma,1)},x_{(i,2)},\ldots,x_{(i,K)}].
\]
As $\sigma$ runs over $G$, $(j_1\pi c\sigma,j_2\pi c\sigma,\ldots,j_{N^{\frac{1}{2}}}\pi c\sigma)$
runs over all permutations of $\{i_1,i_2,\ldots,i_{N^{\frac{1}{2}}}\}$. So to establish equation
(3) for this particular idempotent $e$ it is enough to show that 
\[
\sum_{\sigma\in G}[\ldots,t_{j_1}\pi c\sigma,\ldots,t_{j_2}\pi c\sigma,\ldots,
t_{j_{N^{\frac{1}{2}}}}\pi c\sigma,\ldots]=0
\]
for any given $\pi\in V$ and $c\in C$. Relabelling the free generators of $M$ this is equivalent 
to showing that
\begin{equation}
\sum_{\sigma\in\text{Sym}(N^{\frac{1}{2}})}[\ldots,t_1^{\sigma},\ldots ,t_2^{\sigma},\ldots 
,t_{N^{\frac{1}{2}}}^{\sigma},\ldots]=0.
\end{equation}
where $t_i^{\sigma}=[x_{(i\sigma,1)},x_{(i,2)},\ldots ,x_{(i,K)}]$ for $i=1,2,\ldots ,N^{\frac{1}{2}}$.
Here, and throughout the remainder of this section, all commutators have weight $NK$, and are
multilinear in $\{x_{(i,j)}\,|\,1\leq i\leq N,\,1\leq j\leq K\}$. Extra entries (either free generators
or commutators in free generators) need to be inserted in the ``gaps'' in (4) between the
entries $t_1^{\sigma},t_2^{\sigma},\ldots,t_{N^{\frac{1}{2}}}^{\sigma}$, and it is assumed
that these extra entries remain fixed throughout the sum in (4). Our aim is to prove 
that (4) holds true no matter how these extra entries are inserted.

Next, suppose that $e$ is a primitive idempotent corresponding to a Young
tableau with first column of length at least $N^{\frac{1}{2}}$, and let
\[
e=\frac{1}{m_0}\sum_{\pi\in V,\,\rho\in H}\text{sign}(\pi)\pi\rho
\]
where $V$ is the subgroup of Sym$(N)$ which permutes the entries within
each column of the tableau, and $H$ is the subgroup of Sym$(N)$ which
permutes the entries within each row of the tableau. Pick out the first $N^{\frac{1}{2}}$ 
entries in the first column of the tableau and arrange them in ascending order 
$i_{1}<i_{2}< \ldots <i_{N^{\frac{1}{2}}}$.
Let $G$ be the subgroup of $V$ which fixes
$\{1,2,\ldots,N\}\backslash\{i_{1},i_{2},\ldots,i_{N^{\frac{1}{2}}}\}$ and $C$ be a right
transversal for $G$ in $V$, so that $V=\cup_{c\in C}Gc$. Let $f=\sum_{\sigma \in G}$sign$(\sigma)\sigma$. 
Then
\[
\lbrack\lbrack x_{(1,1)},x_{(1,2)},\ldots,x_{(1,K)}],\ldots ,[x_{(N,1)},x_{(N,2)}\ldots,x_{(N,K)}]]e
\]
is a linear combination elements of the form $[t_1,t_2,\ldots,t_N]fc\rho$ with $c\in C$ and $\rho\in H$.
furthermore
\[
[t_1,t_2,\ldots,t_N]fc\rho
\]
\[
=\sum_{\sigma \in G}\text{sign}(\sigma )[t_1\sigma,\ldots,t_{i_1}\sigma,\ldots ,t_{i_2}\sigma,\ldots ,t_{i_{N^{\frac{1}{2}}}}\sigma,\ldots,t_N\sigma ]c\rho
\]
where $t_i\sigma = t_i$ if $i \notin \{i_1,i_2,\ldots,i_{N^{\frac{1}{2}}}\}$, and
$t_{i_j}^{\sigma}=[x_{(i_j\sigma,1)},x_{(i_j,2)}\ldots,x_{(i_j,K)}]$ for $j=1,2,\ldots ,N^{\frac{1}{2}}$.

So, as above, to show that
\[
\lbrack\lbrack x_{(1,1)},x_{(1,2)},\ldots,x_{(1,K)}],\ldots ,[x_{(N,1)},x_{(N,2)}\ldots,x_{(N,K)}]]e=0
\]
it is sufficient to show that
\begin{equation}
\sum_{\sigma\in\text{Sym}(N^{\frac{1}{2}})}\text{sign}(\sigma )[\ldots,t_1^{\sigma},\ldots ,t_2^{\sigma},\ldots ,t_{N^{\frac{1}{2}}}^{\sigma},\ldots]=0.
\end{equation}
where $t_i^{\sigma}=[x_{(i\sigma,1)},x_{(i,2)},\ldots ,x_{(i,K)}]$ for $i=1,2,\ldots ,N^{\frac{1}{2}}$.

If we denote the sum in (4) as $\sum^{+}$ and the sum in (5) as $\sum^{-}$ then we see that to
establish equation (2) it is sufficient to prove that
\begin{equation}
\sum_{\sigma\in\text{Sym}(N^{\frac{1}{2}})}^{\varepsilon}[\ldots,t_1^{\sigma},\ldots ,t_2^{\sigma},\ldots ,t_{N^{\frac{1}{2}}}^{\sigma},\ldots]=0
\end{equation}
for $\varepsilon=+$ and also for $\varepsilon=-$.

We now let $F$Sym$(N^{\frac{1}{2}})$ act on $M$, permuting the free generators
$x_{(1,2)},x_{(2,2)},\ldots,x_{(N^{\frac{1}{2}},2)}$. To establish (6) it is enough to show that
\[
\sum_{\sigma\in\text{Sym}(N^{\frac{1}{2}})}^{\varepsilon}[\ldots,t_1^{\sigma},\ldots ,t_2^{\sigma},\ldots ,t_{N^{\frac{1}{2}}}^{\sigma},\ldots]e=0
\]
for every primitive idempotent $e \in F$Sym$(N^{\frac{1}{2}})$. Any Young tableau in $F$Sym$(N^{\frac{1}{2}})$ will either
have first row of length at least $N^{\frac{1}{4}}$ or first column with length at least $N^{\frac{1}{4}}$.

First consider the case when $e$ corresponds to a Young tableau with first row of length at least $N^{\frac{1}{4}}$,
and let
\[
e=\frac{1}{m_0}\sum_{\pi\in V,\,\rho\in H}\text{sign}(\pi)\pi\rho
\]
where $V$ is the subgroup of Sym$(N^{\frac{1}{2}})$ which permutes the entries within
each column of the tableau, and $H$ is the subgroup of Sym$(N^{\frac{1}{2}})$ which
permutes the entries within each row of the tableau. Pick out the first $N^{\frac{1}{4}}$ 
entries in the first row of the tableau and arrange them in ascending order 
$i_{1}<i_{2}< \ldots <i_{N^{\frac{1}{4}}}$.
Let $G$ be the subgroup of $H$ which fixes
$\{1,2,\ldots,N^{\frac{1}{2}}\}\backslash\{i_{1},i_{2},\ldots,i_{N^{\frac{1}{4}}}\}$ and $C$ be a left
transversal for $G$ in $H$, so that $H=\cup_{c\in C}cG$. Let $f=\sum_{\tau\in G}\tau$. Then
\[
\sum_{\sigma\in\text{Sym}(N^{\frac{1}{2}})}^{\varepsilon}[\ldots,t_1^{\sigma},\ldots ,t_2^{\sigma},\ldots ,t_{N^{\frac{1}{2}}}^{\sigma}\ldots]e
\]
is a linear combination of terms of the form
\[
\sum_{\sigma\in\text{Sym}(N^{\frac{1}{2}})}^{\varepsilon}[\ldots,t_1^{\sigma},\ldots ,t_2^{\sigma},\ldots ,t_{N^{\frac{1}{2}}}^{\sigma},\ldots]\pi cf
\]
with $\pi\in V$ and $c\in C$. For fixed $\pi\in V$ and $c\in C$ let
\[
\{j_1\pi c,j_2\pi c,\ldots ,j_{N^{\frac{1}{4}}}\pi c\}=\{i_{1},i_{2},\ldots,i_{N^{\frac{1}{4}}}\}
\]
with $j_1<j_2<\ldots <j_{N^{\frac{1}{4}}}$. Then
\[
\sum_{\sigma\in\text{Sym}(N^{\frac{1}{2}})}^{\varepsilon}[\ldots,t_1^{\sigma},\ldots ,t_2^{\sigma},\ldots ,t_{N^{\frac{1}{2}}}^{\sigma},\ldots]\pi cf
\]
\[
=\sum_{\sigma\in\text{Sym}(N^{\frac{1}{2}})}^{\varepsilon}\sum_{\tau \in G}[\ldots,t_1^{\sigma},\ldots ,t_{j_1}^{\sigma},\ldots ,t_{j_{N^{\frac{1}{4}}}}^{\sigma},
\ldots ,t_{N^{\frac{1}{2}}}^{\sigma},\ldots]\pi c\tau
\]
\[
=\sum_{\sigma\in\text{Sym}(N^{\frac{1}{2}})}^{\varepsilon}\sum_{\tau \in G}[\ldots,t_1^{\sigma}\pi c,\ldots ,t_{j_1}^{\sigma}\pi c\tau,\ldots ,
t_{j_{N^{\frac{1}{4}}}}^{\sigma}\pi c\tau,\ldots ,t_{N^{\frac{1}{2}}}^{\sigma}\pi c,\ldots]
\]
since if $\tau \in G$ then $\tau$ fixes $t_j^{\sigma}\pi c$ unless $j \in \{j_1,j_2,\ldots ,j_{N^{\frac{1}{4}}}\}$.

Now if $1\leq r\leq N^{\frac{1}{4}}$ then
\[
t_{j_r}^{\sigma}\pi c\tau=[x_{(j_r\sigma,1)},x_{(j_r\pi c\tau,2)},x_{(j_r,3)}\ldots ,x_{(j_r,K)}].
\]
As $\tau$ ranges over $G$, $(j_1\pi c\tau,j_2\pi c\tau,\ldots ,j_{N^{\frac{1}{4}}}\pi c\tau)$ ranges over all
possible permutations of $\{i_1,i_2,\ldots ,i_{N^{\frac{1}{4}}}\}$. And as $\sigma$ ranges over Sym$(N^{\frac{1}{2}})$,
$(j_1\sigma,j_2\sigma,\ldots ,j_{N^{\frac{1}{4}}}\sigma)$ ranges over all possible permutations of subsets $S$
where $S$ ranges over all possible $N^{\frac{1}{4}}$ element subsets of $\{1,2,\ldots ,N^{\frac{1}{2}}\}$.
Fix on one particular $N^{\frac{1}{4}}$ element subset $S$ of $\{1,2,\ldots ,N^{\frac{1}{2}}\}$,
and pick $\sigma_0$ such that $\{j_1\sigma_0,j_2\sigma_0,\ldots ,j_{N^{\frac{1}{4}}}\sigma_0\}=S$.
Let $A$ be the group of all permutations of $S$ and let $B$ be the group of all permutations of
$\{1,2,\ldots ,N^{\frac{1}{2}}\}\backslash S$. Then any permutation $\sigma \in \,$Sym$(N^{\frac{1}{2}})$ 
which the property that $\{j_1\sigma,j_2\sigma,\ldots ,j_{N^{\frac{1}{4}}}\sigma\}=S$ can be
written uniquely in the form $\sigma =\sigma_0ab$ with $a \in A$ and $b \in B$.
So if we pick out the terms in
\[
\sum_{\sigma\in\text{Sym}(N^{\frac{1}{2}})}^{\varepsilon}\sum_{\tau \in G}[\ldots,t_1^{\sigma}\pi c,\ldots ,t_{j_1}^{\sigma}\pi c\tau,\ldots ,
t_{j_{N^{\frac{1}{4}}}}^{\sigma}\pi c\tau,\ldots ,t_{N^{\frac{1}{2}}}^{\sigma}\pi c,\ldots]
\]
where $\{j_1\sigma,j_2\sigma,\ldots ,j_{N^{\frac{1}{4}}}\sigma\}=S$ then we obtain
\[
\pm \sum_{a \in A}^{\varepsilon}\sum_{b \in B}^{\varepsilon}\sum_{\tau \in G}[\ldots,t_1^{\sigma_0b}\pi c,\ldots ,t_{j_1}^{\sigma_0a}\pi c\tau,\ldots ,
t_{j_{N^{\frac{1}{4}}}}^{\sigma_0a}\pi c\tau,\ldots ,t_{N^{\frac{1}{2}}}^{\sigma_0b}\pi c,\ldots].
\]
This is a sum of $|B|$ terms of the form
\[
\pm \sum_{a \in A}^{\varepsilon}\sum_{\tau \in G}[\ldots,t_1^{\sigma_0b}\pi c,\ldots ,t_{j_1}^{\sigma_0a}\pi c\tau,\ldots ,
t_{j_{N^{\frac{1}{4}}}}^{\sigma_0a}\pi c\tau,\ldots ,t_{N^{\frac{1}{2}}}^{\sigma_0b}\pi c,\ldots],
\]
one for each $b \in B$. To show that
\[
\sum_{\sigma\in\text{Sym}(N^{\frac{1}{2}})}^{\varepsilon}[\ldots,t_1^{\sigma},\ldots ,t_2^{\sigma},\ldots ,t_{N^{\frac{1}{2}}}^{\sigma},\ldots]e=0
\]
it is enough to show that each of these individual
sums is zero. Simplifying the notation, this is equivalent to showing that
\[
\sum_{\sigma\in\text{Sym}(N^{\frac{1}{4}})}^{\varepsilon}\sum_{\tau \in\text{Sym}(N^{\frac{1}{4}})}
[\ldots,t_1^{(\sigma,\tau)},\ldots,t_2^{(\sigma,\tau)},\ldots,t_{N^{\frac{1}{4}}}^{(\sigma,\tau)},\ldots]=0
\]
where $t_i^{(\sigma,\tau)}=[x_{(i\sigma,1)},x_{(i\tau,2)},x_{(i,3)},\ldots,x_{(i,K)}]$ for $i=1,2,\ldots,N^{\frac{1}{4}}$.

Similarly, if $e$ is an idempotent corresponding to a Young tableau with first column of length at least $N^{\frac{1}{4}}$,
then to show that
\[
\sum_{\sigma\in\text{Sym}(N^{\frac{1}{2}})}^{\varepsilon}[\ldots,t_1^{\sigma},\ldots ,t_2^{\sigma},\ldots ,t_{N^{\frac{1}{2}}}^{\sigma},\ldots]e=0
\]
it is sufficient to show that
\[
\sum_{\sigma\in\text{Sym}(N^{\frac{1}{4}})}^{\varepsilon}\sum_{\tau \in\text{Sym}(N^{\frac{1}{4}})}^-
[\ldots,t_1^{(\sigma,\tau)},\ldots,t_2^{(\sigma,\tau)},\ldots,t_{N^{\frac{1}{4}}}^{(\sigma,\tau)},\ldots]=0
\]
So to establish equation (6) it is enough to show that
\[
\sum_{\sigma\in\text{Sym}(N^{\frac{1}{4}})}^{\varepsilon}\sum_{\tau \in\text{Sym}(N^{\frac{1}{4}})}^{\eta}
[\ldots,t_1^{(\sigma,\tau)},\ldots,t_2^{(\sigma,\tau)},\ldots,t_{N^{\frac{1}{4}}}^{(\sigma,\tau)},\ldots]=0
\]
for $\eta = +$ and for $\eta = -$.

We next let $F$Sym$(N^{\frac{1}{4}})$ act
on $M$, permuting the free generators $x_{(1,3)},x_{(2,3)},\ldots,x_{(N^{\frac{1}{4}},3)}$, 
and so on. Continuing in this manner for $K$ steps
we eventually see that if we let $R=Tk$, then it is enough to prove that for
every choice of $\varepsilon_{1},\varepsilon_{2},\ldots \varepsilon_{K}$ equal
to $+$ or equal to $-$,
\[
\sum_{\sigma_{1}\in\text{Sym}(R)}^{\varepsilon_{1}}\sum_{\sigma_{2}%
\in\text{Sym}(R)}^{\varepsilon_{2}}\ldots\sum_{\sigma_{K}\in\text{Sym}%
(R)}^{\varepsilon_{K}}[\ldots,t_{1}^{(\sigma_1,\ldots,\sigma_K)},\ldots,t_{2}^{(\sigma_1,\ldots,\sigma_K)}
,\ldots,t_{R}^{(\sigma_1,\ldots,\sigma_K)},\ldots]=0
\]
where
\[
t_{i}^{(\sigma_1,\ldots,\sigma_K)}=[x_{(i\sigma_{1},1)},x_{(i\sigma_{2},2)},\ldots,x_{(i\sigma_{K},K)}]
\]
for $i=1,2,\ldots,R$. We alter the notation slightly and rewrite
the left hand side of this equation as%
\begin{equation}
\sum_{\sigma_{1}\in S_{1}}^{\varepsilon_{1}}\sum_{\sigma_{2}\in S_{2}%
}^{\varepsilon_{2}}\ldots\sum_{\sigma_{K}\in S_{K}}^{\varepsilon_{K}}%
[\ldots,t_{1},\ldots,t_{2},\ldots,t_{R},\ldots]\sigma_{1}\sigma_{2}\ldots\sigma_{K}%
\end{equation}
where $S_{1}$ is a copy of Sym$(R)$ which permutes the free generators
$x_{(1,1)},x_{(2,1)},\ldots,x_{(R,1)}$ (so $S_{1}$ permutes generators rather
than indices), where $S_{2}$ is a copy of Sym$(R)$ which permutes the free
generators $x_{(1,2)},x_{(2,2)},\ldots,x_{(R,2)}$, and so on, and where
$t_{i}=[x_{(i,1)},x_{(i,2)},\ldots,x_{(i,K)}]$ for $i=1,2,\ldots,R$.

We now fix a choice of $+$ or $-$ for each of $\varepsilon_{1},\varepsilon
_{2},\ldots,\varepsilon_{K}$ and apply Lemma 1 from Section 3. We let $L$ be
the Lie subring of $M$ generated by the free generators $x_{1},x_{2}%
,\ldots,x_{K}$. We turn $L$ into a $\mathbb{Z}_{2}$-graded Lie algebra
$L=L_{0}\oplus L_{1}$ letting $x_{i}\in L_{0}$ if $\varepsilon_{i}=+$, and
letting $x_{i}\in L_{1}$ if $\varepsilon_{i}=-$. We let $C$ be the set of all
possible left-normed commutators $[x_{i_{1}},x_{i_{2}},\ldots,x_{i_{r}}]$
where $r\geq1$ and $i_{1},i_{2},\ldots,i_{r}\in\{1,2,\ldots,K\}$. Then
$C=C_{0}\cup C_{1}$, where $C_{0}\subset L_{0}$ and $C_{1}\subset L_{1}$. Let
$J$ be the ideal of $L$ generated by all possible elements $[c_{1}%
,c_{2},\ldots,c_{m}]$ with $c_{i}\in C_{0}$ for $i=1,2,\ldots,m$. Then $L/J$
satisfies the hypothesis of Lemma 1, and so $[x_{1},x_{2},\ldots,x_{K}]\in J$.
This implies that $[x_{1},x_{2},\ldots,x_{K}]$ is a finite linear combination
$\sum_{r=1}^{t}\alpha_{r}u_{r}$ ($\alpha_{r}\in F$) of multilinear terms
$u_{r}$ of weight $K$ of the form $[[c_{1},c_{2},\ldots,c_{m}],a_{1},a_{2},\ldots,a_{q}]$
with $c_{i}\in C_{0}$ for $i=1,2,\ldots,m$ and with $a_{1},a_{2},\ldots
,a_{q}\in \{x_1,x_2,\ldots,x_K\}$ $(q\geq0)$. We chose $T$ in Section 3 so $t\leq T$. For each
$i=1,2,\ldots,R$ we let $\theta_{i}$ be the endomorphism of $M$ mapping
$x_{j}$ to $x_{(i,j)}$ for $j=1,2,\ldots,K$, so that%
\[
t_{i}=[x_{(i,1)},x_{(i,2)},\ldots,x_{(i,K)}]=[x_{1},x_{2},\ldots,x_{K}%
]\theta_{i}=\sum_{r=1}^{t}\alpha_{r}u_{r}\theta_{i}.
\]
for $i=1,2,\ldots,R$. We substitute this sum for each $t_{i}$ in (7), expand,
and obtain a linear combination of expressions%
\[
\sum_{\sigma_{1}\in S_{1}}^{\varepsilon_{1}}\sum_{\sigma_{2}\in S_{2}%
}^{\varepsilon_{2}}\ldots\sum_{\sigma_{K}\in S_{K}}^{\varepsilon_{K}}%
[\ldots,u_{r_{1}}\theta_{1},\ldots,u_{r_{2}}\theta_{2},\ldots,u_{r_{R}}\theta
_{R},\ldots]\sigma_{1}\sigma_{2}\ldots\sigma_{K}%
\]
over all possible choices of $1\leq r_{1},r_{2},\ldots,r_{R}\leq t$. Since
$R=Tk\geq tk$, for any such choice of $r_{1},r_{2},\ldots,r_{R}$ there must be
some index, $r$ say, which appears at least $k$ times in the sequence. Suppose
that $r_{i}=r$ for $i=i_{1},i_{2},\ldots,i_{k}$. Then%
\begin{align*}
&  \sum_{\sigma_{1}\in S_{1}}^{\varepsilon_{1}}\sum_{\sigma_{2}\in S_{2}%
}^{\varepsilon_{2}}\ldots\sum_{\sigma_{K}\in S_{K}}^{\varepsilon_{K}}%
[\ldots,u_{r_{1}}\theta_{1},\ldots,u_{r_{2}}\theta_{2},\ldots,u_{r_{R}}\theta
_{R}\ldots]\sigma_{1}\sigma_{2}\ldots\sigma_{K}\\
&  =\sum_{\sigma_{1}\in S_{1}}^{\varepsilon_{1}}\sum_{\sigma_{2}\in S_{2}%
}^{\varepsilon_{2}}\ldots\sum_{\sigma_{K}\in S_{K}}^{\varepsilon_{K}}%
[\ldots,u_{r}\theta_{i_{1}},\ldots,u_{r}\theta_{i_{2}},\ldots,u_{r}%
\theta_{i_{k}},\ldots]\sigma_{1}\sigma_{2}\ldots\sigma_{K}.
\end{align*}
If $u_{r}=[[c_{1},c_{2},\ldots,c_{m}],a_{1},a_{2},\ldots,a_{q}]$ then
\[
[\ldots,u_{r}\theta_{i_{1}},\ldots,u_{r}\theta_{i_{2}},\ldots,u_{r}\theta_{i_{k}},\ldots]
\]
is a linear combination of multilinear commutators of the form
\[
[\ldots,[c_{1},c_{2},\ldots,c_{m}]\theta_{i_{1}},\ldots,[c_{1},c_{2},\ldots,c_{m}]\theta_{i_{2}},
\ldots,[c_{1},c_{2},\ldots,c_{m}]\theta_{i_{k}},\ldots].
\]
so to show that (7) equals zero it is sufficient to show that
\begin{equation}
\sum_{\sigma_{1}\in S_{1}}^{\varepsilon_{1}}\sum_{\sigma_{2}\in S_{2}%
}^{\varepsilon_{2}}\ldots\sum_{\sigma_{K}\in S_{K}}^{\varepsilon_{K}}%
[\ldots,[c_{1},c_{2},\ldots,c_{m}]\theta_{i_{1}},\ldots,[c_{1},c_{2}%
,\ldots,c_{m}]\theta_{i_{k}},\ldots]\sigma_{1}\sigma_{2}\ldots\sigma_{K}=0.
\end{equation}
We show that equation (1) implies equation (8).
(This will complete our proof of Zel'manov's theorem.)

We let $\sigma_{i}\in S_{i}$ for $i=1,2,\ldots,K$, and we let $\sigma
=\sigma_{1}\sigma_{2}\ldots\sigma_{K}$, and we consider the single term%
\[
\pm\lbrack\ldots,[c_{1},c_{2},\ldots,c_{m}]\theta_{i_{1}},\ldots,[c_{1}%
,c_{2},\ldots,c_{m}]\theta_{i_{2}},\ldots,[c_{1},c_{2},\ldots,c_{m}%
]\theta_{i_{k}},\ldots]\sigma
\]
from the sum in (8). Pick $i,j \in \{i_1,i_2,\ldots,i_k\}$ ($i<j$), and consider the action of
$\sigma$ on $[c_{1},c_{2},\ldots,c_{m}]\theta_{i}$ and $[c_{1},c_{2}%
,\ldots,c_{m}]\theta_{j}$.%
\[
\lbrack c_{1},c_{2},\ldots,c_{m}]\theta_{i}\sigma=[c_{1}\theta_{i}\sigma
,c_{2}\theta_{i}\sigma,\ldots,c_{m}\theta_{i}\sigma].
\]
Suppose that $c_{1}=[x_{k_{1}},x_{k_{2}},\ldots,x_{k_{q}}]$ $(q\geq1)$. Then%
\[
c_{1}\theta_{i}\sigma=[x_{(i,k_{1})},x_{(i,k_{2})},\ldots,x_{(i,k_{q})}%
]\sigma=[x_{(i,k_{1})}\sigma_{k_{1}},x_{(i,k_{2})}\sigma_{k_{2}}%
,\ldots,x_{(i,k_{q})}\sigma_{k_{q}}].
\]
(We are using the fact that $\sigma_{s}$ fixes $x_{(i,j)}$ unless $s=j$.)
Similarly
\[
\lbrack c_{1},c_{2},\ldots,c_{m}]\theta_{j}\sigma=[c_{1}\theta_{j}\sigma
,c_{2}\theta_{j}\sigma,\ldots,c_{m}\theta_{j}\sigma],
\]
and%
\[
c_{1}\theta_{j}\sigma=[x_{(j,k_{1})},x_{(j,k_{2})},\ldots,x_{(j,k_{q})}%
]\sigma=[x_{(j,k_{1})}\sigma_{k_{1}},x_{(j,k_{2})}\sigma_{k_{2}}%
,\ldots,x_{(j,k_{q})}\sigma_{k_{q}}].
\]
Now let $\tau_{1}$ be the transposition in $S_{k_{1}}$ which swaps
$x_{(i,k_{1})}\sigma_{k_{1}}$ and $x_{(j,k_{1})}\sigma_{k_{1}}$, let $\tau
_{2}$ be the transposition in $S_{k_{2}}$which swaps $x_{(i,k_{2})}%
\sigma_{k_{2}}$ and $x_{(j,k_{2})}\sigma_{k_{2}}$, and so on. Note that the
sign attached to $\tau_{1}$ in equation (8) is $\varepsilon_{k_1}$, and that the sign attached to
$\tau_{2}$ is $\varepsilon_{k_2}$, and so on.
Let $\tau
=\tau_{1}\tau_{2}\ldots\tau_{q}$. Note that since $c_{1}\in L_{0}\,$, the sign
attached to $\tau$ is $+$. So%
\[
\lbrack\ldots,[c_{1},c_{2},\ldots,c_{m}]\theta_{i_{1}},\ldots,[c_{1}%
,c_{2},\ldots,c_{m}]\theta_{i_{2}},\ldots,[c_{1},c_{2},\ldots,c_{m}%
]\theta_{i_{k}},\ldots]\sigma
\]
and%
\[
\lbrack\ldots,[c_{1},c_{2},\ldots,c_{m}]\theta_{i_{1}},\ldots,[c_{1}%
,c_{2},\ldots,c_{m}]\theta_{i_{2}},\ldots,[c_{1},c_{2},\ldots,c_{m}%
]\theta_{i_{k}},\ldots]\sigma\tau
\]
are two terms from the sum in (8) with the same sign. Picking out the action
of $\sigma$ and $\sigma\tau$ on $[c_{1},c_{2},\ldots,c_{m}]\theta_{i}$ and
$[c_{1},c_{2},\ldots,c_{m}]\theta_{j}$ we can write these two expressions as%
\[
\lbrack\ldots,[c_{1}\theta_{i}\sigma,c_{2}\theta_{i}\sigma,\ldots,c_{m}%
\theta_{i}\sigma],\ldots,[c_{1}\theta_{j}\sigma,c_{2}\theta_{j}\sigma
,\ldots,c_{m}\theta_{j}\sigma],\ldots]
\]
and%
\[
\lbrack\ldots,[c_{1}\theta_{i}\sigma\tau,c_{2}\theta_{i}\sigma,\ldots
,c_{m}\theta_{i}\sigma],\ldots,[c_{1}\theta_{j}\sigma\tau,c_{2}\theta
_{j}\sigma,\ldots,c_{m}\theta_{j}\sigma],\ldots]
\]
where corresponding unspecified entries are the same in these two commutators.
Our choice of $\tau$ implies that $c_{1}\theta_{i}\sigma\tau=c_{1}\theta
_{j}\sigma$ and $c_{1}\theta_{j}\sigma\tau=c_{1}\theta_{i}\sigma$. So $\tau$ swaps
the two entries $c_{1}\theta_{i}\sigma$ and $c_{1}\theta_{j}\sigma$, and leaves 
everything else fixed.

Now let $\sigma_{i}$ range over all of $S_{i}$ for all of $i=1,2,\ldots,K$ and
write (8) as%
\[
\sum_{\sigma}\pm\lbrack\ldots,[c_{1}\theta_{i_{1}}\sigma,c_{2}\theta_{i_{1}%
}\sigma,\ldots,c_{m}\theta_{i_{1}}\sigma],\ldots,[c_{1}\theta_{i_{k}}%
\sigma,c_{2}\theta_{i_{k}}\sigma,\ldots,c_{m}\theta_{i_{k}}\sigma],\ldots]
\]
where the unspecified entries are also acted on by $\sigma$. Then we have
shown that this sum is symmetric in the entries $c_{1}\theta_{i_{1}}\sigma$,
$c_{1}\theta_{i_{2}}\sigma$, \ldots, $c_{1}\theta_{i_{k}}\sigma$. Similarly we
see that this expression is symmetric in $c_{j}\theta_{i_{1}}\sigma$,
$c_{j}\theta_{i_{2}}\sigma$, \ldots, $c_{j}\theta_{i_{k}}\sigma$ for all
$j=1,2,\ldots,m$. So equation (1) implies that this sum is zero.

\end{document}